\documentclass[journal,10pt,onecolumn,draftclsnofoot]{IEEEtran}  

\IEEEoverridecommandlockouts                              

\usepackage{amsmath, amsfonts, dsfont, amssymb, mathrsfs, setspace, graphicx, epsfig, amsthm,url,float,framed,circuitikz, color,subcaption,epstopdf}
\usepackage{multicol}
\usepackage{hyperref}
\usepackage{overpic}
\usetikzlibrary{shapes.gates.logic.US,trees,positioning,arrows,quotes}

\usepackage{pdfpages}
\usepackage{xcolor}
\usepackage{subcaption}

\let\labelindent\relax
\usepackage{enumitem}

\DeclareMathOperator{\diag}{diag}

\newcommand{\1}{\mathbf{1}}
\newcommand{\0}{\mathbf{0}}

\newtheorem{lemma}{Lemma}

\newtheorem{theorem}{Theorem}
\newtheorem{assumption}{Assumption}

\makeatletter
\usepackage[noadjust]{cite}
\makeatother

\newcommand\undermat[2]{%
  \makebox[0pt][l]{$\smash{\underbrace{\phantom{%
    \begin{matrix}#2\end{matrix}}}_{\text{$#1$}}}$}#2}

\title{Analysis and Estimation of Networked SIR \& SEIR Models with Transportation Networks}

\author{Damir Vrabac, 
Raphael Stern, 
        Philip~E.~Par\'e
*\thanks{
* Damir Vrabac is with the School of Electrical Engineering at Stanford University. 
Raphael Stern is with the Department of Civil, Environmental, and Geo- Engineering at the University of Minnesota. 
Philip E. Par\'{e} is with the School of Electrical and Computer Engineering at Purdue University.
E-mails: dvrabac@stanford.edu, rstern@umn.edu, philpare@purdue.edu. This material is based upon work supported by the National Science Foundation under Grant No. CNS-2028946 (R.S.) and Grant No. CNS-2028738 (P.E.P.), as well as the C3.ai Digital Transformation Institute sponsored by C3.ai Inc. and the Microsoft Corporation (P.E.P.).
}}
\begin{document}
\maketitle

\begin{abstract}
In this paper we present the discrete-time networked SIR and SEIR models and present assumptions under which they are well defined. We analyze the limiting behavior of the models and present necessary and sufficient conditions for estimating the spreading parameters from data. We illustrate these results via simulation.
\end{abstract}

\section{Introduction}

Given recent outbreaks, it is critical to be able to quickly track the spread of the virus and understand the mechanisms that are enabling their propagation. While the mode of transmission of recent novel viruses is not exactly known, human-to-human interaction appears to be a main factor~\cite{lewnard2020scientific}. A key component for transmission is the underlying transportation network, which acts as a propagator of the virus within and between communities.

In this work we extend the commonly used 
SIR
~\cite{sir} and 
SEIR
~\cite{seir} models for viral spread to consider spread over the network in the context of human interaction and transportation. We model the proportion of people in each county who have not been infected ($S$), those who have been infected but have not been confirmed via a test ($E$), test-confirmed infected cases ($I$), and those who have either recovered or died from the virus ($R$) 
and show that we are able to accurately model the evolution of such a virus, as well as how to recover the proper model parameter values from time series data of infections and recoveries. 

Previous work has explored the networked SIR models \cite{mei2017epidemic,sir_closed_loop}. Specifically, the authors in \cite{mei2017epidemic} provide stability properties and asymptotic convergence, as well as a novel algorithm to compute the asymptomatic state of the network SIR system. More recently, the SEIR model has become popular for modeling epidemic spread (e.g., \cite{zhou2011analysis}). 
The model has also been extended to account for quarantine~\cite{groendyke2020modifying} and asymptomatic transmission~\cite{arcede2020accounting}. 
We go beyond prior work by analyzing the limiting behavior of the network SEIR model and present estimation results of the spread parameters.

The remainder of the article is outlined as follows. The networked SEIR model is introduced in Section~\ref{sec:models} and its limiting behavior is discussed in Section~\ref{sec:analysis}. Parameter estimation for the networked SIR and SEIR models is discussed in Section~\ref{sec:parameters} and demonstrated in Section~\ref{sec:sim}.

\subsection{Notation}

Given a vector $x$, the transpose is indicated by $x^\top$, $\bar{x}$ is the average of its entries, and $\diag(\cdot)$ 
is  
a diagonal matrix with the argument on the diagonal. We use $\0$ and $\1$ to denote a vector or matrix of zeros and ones, respectively, of the appropriate dimensions. We define a directed graph $\mathcal{G} = (\mathcal{V},\mathcal{E},w)$, where $\mathcal{V}$ is the set of nodes, $ \mathcal{E} \subseteq \mathcal{V}\times \mathcal{V} $ is the set of edges, and $w : \mathcal{E} \rightarrow \mathbb{R}^+$ is a function
mapping directed edges to their weightings, with $\mathbb{R}^+$ being the set of positive real values.
Given $\mathcal{G}$, we denote an edge from
node $i \in \mathcal{V}$ to node $j\in \mathcal{V}$ by $(i,j)$. We say node $i \in \mathcal{V}$ is a neighbor of node $j \in \mathcal{V}$ if and only if $(i,j) \in \mathcal{E}$, and denote the neighbors of node $j$ as $\mathcal{N}_j$. We denote the weighted adjacency matrix associated with $\mathcal{G}$ as $A$ with the nonzero entry $a_{ji}$ indicating the strength of edge~$(i,j)$ as given by~$w$. 

\section{Networked SIR \& SEIR Models}\label{sec:models}

Here we introduce the discrete-time networked SIR and SEIR models. In each case we assume that the virus spreads over $\mathcal{G} = (\mathcal{V},\mathcal{E},w)$ with adjacency matrix $A$. 
Each node in $\mathcal{V}$ can be interpreted as a single individual or a subpopulation and the states are interpreted as probabilities or proportions, respectively, referred to as levels from here on out. 


\subsection{Networked SIR Model}\label{subsec:nSIR}


For an SIR process spreading over 
$\mathcal{G}$, node $i$'s levels of susceptibility $ s^{k}_i$, infection $ p^{k}_i$,  and recovery $ r^{k}_i$ evolve as 
\begin{subequations}\label{eq:nsir}
\begin{align}
    s_i^{k+1} & = s_i^k - h s_i^k \Bigg( \beta_i\sum_{j\in \mathcal{N}_i} a_{ij} p_j^k 
    \Bigg), \\
    p^{k+1}_i &= p^{k}_i + h\Bigg(s_i^k 
    \beta_i \sum_{j\in \mathcal{N}_i}a_{ij} p_j^k
    - \gamma_i p_i^k \Bigg), \label{eq:pidiscrete} \\
    r^{k+1}_i &= r^{k}_i + h \gamma_i p_i^k, \label{eq:ridiscrete}
\end{align}
\end{subequations}
%
where $k$ is the time step, $h$ is the sampling parameter, and  $\beta_i$ 
and $\gamma_i$ 
are 
node $i$'s infection and recovery parameters, respectively. The SIR discrete-time model can be expressed in matrix form as follows
\begin{subequations}
\begin{align}
    p^{k+1} &= p^{k} + h \left( (I - P^k - R^k) 
    B
    A 
    - \gamma \right) p^k, \label{eq:pdiscrete} \\
    r^{k+1} &= r^{k} + h \gamma p^k , \label{eq:rdiscrete}
\end{align}
\end{subequations}
%
where $P^k = \diag(p_i^k)$, $R^k = \diag(r_i^k)$, $B = \diag(\beta_i)$,  and $\gamma = \diag(\gamma_i^k)$. 
For this model to be well-defined we need the following assumptions.
\begin{assumption}\label{assm:0}
For all $i \in [n]$, we have $0<h \gamma_i < 1$ and $h \sum_{j\in \mathcal{N}_i} \beta_ia_{ij} < 1$. 
\end{assumption}
\begin{lemma} {\cite{sir_closed_loop}}
Consider the model in \eqref{eq:nsir} under Assumption \ref{assm:0}. Suppose $s_i^0, p_i^0,$ $r_i^0 \in [0,1]$ and $s_i^0 + p_i^0 + r_i^0 = 1$ for all $i \in [n] $. Then, for all $k \geq 0$ and $i \in [n]$, $s_i^k, p_i^k, r_i^k \in [0,1]$ and $s_i^k + p_i^k + r_i^k = 1$.
\end{lemma}

\subsection{Networked SEIR Model}\label{sec:nSEIR}

For the  discrete-time SEIR model spreading over 
$\mathcal{G}$ 
with $s_i^k = 1 -e_i^k- p_i^k -r_i^k$, 
the exposed level of  $e^{k}_i$ and other states evolve as
\begin{subequations}\label{eq:nseir}
\begin{align}
    s_i^{k+1} & = s_i^k - h s_i^k \iota^k_i
    , \label{eq:seir_si} \\
    e^{k+1}_i &= e^{k}_i + h s_i^k 
    \iota^k_i
    - h\sigma_i e_i^k 
    , \label{eq:seir_ei} \\
    p^{k+1}_i &= p^{k}_i + h(\sigma_i e_i^k - \gamma_i p_i^k), \label{eq:seir_pi} \\
    r^{k+1}_i &= r^{k}_i + h\left(\gamma_i p_i^k \right), \label{eq:seir_ri}
\end{align}
\end{subequations}
%
 where 
 \begin{equation}\label{eq:iota1}
     \iota^k_i = \beta^E_i\sum_{j\in \mathcal{N}_i} a_{ij} e_j^k + \beta_i\sum_{j\in \mathcal{N}_i} a_{ij} p_j^k, 
 \end{equation}
 $\beta^E_i$ is infection parameter 
 associated with the $e^k_i$ state, $\sigma_i$ captures the rate at which the exposed become 
 confirmed infected cases. 

For transportation model we define
\begin{equation}\label{eq:iota2}
     \check{\iota}^k_i =  \iota^k_i+
     \sum_{l\in\mathcal{L}}\left(\check{\beta}^{E,l}_i\sum_{j\in \mathcal{N}_i} \check{a}^l_{ij} e_j^k + \check{\beta}^l_i\sum_{j\in \mathcal{N}_i} \check{a}^l_{ij} p_j^k\right), 
 \end{equation}
 where $\mathcal{L}$ is the set of transportation networks, $\check{a}^l_{ij}$ is the  transportation network, and $\check{\beta}^{E,l}_i$ and $\check{\beta}^l_i$ are the corresponding infection rates, for the $l$th transportation network.
 
 The SEIR discrete-time model can also be expressed in matrix form as follows
 \begin{subequations}\label{eq:seir_dis_matrix}
\begin{align}
    e^{k+1} &= e^{k} + h \left( 
    S ^k
    (B^E A e^k + BA p^k) 
    - \sigma e^k \right) , \label{eq:ediscrete} \\
    p^{k+1} &= p^{k} + h \left( \sigma e^k - \gamma p^k \right) , \label{eq:pdiscrete_seir} \\
    r^{k+1} &= r^{k} + h \left( \gamma p^k  \right), \label{eq:rdiscrete_seir}
\end{align}
\end{subequations}
%
where $S^k = \diag(s_i^k)$, 
$B^E = {\rm diag}(\beta^E_i)$, 
and $\sigma = {\rm diag}(\sigma_i)$.

For the discrete-time SEIR model to be well-defined we need the following assumptions.
\begin{assumption}\label{assm:seir}
For all $i \in [n]$, we have $0 < h \gamma_i < 1$, $0 < h\sigma_i \leq 1$, and $0 \leq h(\beta^E_i + \beta_i)\sum_{j\in \mathcal{N}_i}
a_{ij} < 1$. $\beta^E_i, \beta_i,a_{ij} \geq 0$ for all $i,j \in [n]$. 
\end{assumption}
%
%


\begin{lemma} \label{lem:seir}
Consider the model in \eqref{eq:nseir}-\eqref{eq:iota1} under Assumption \ref{assm:seir}. Suppose $s_i^0, e_i^0, p_i^0,$ $r_i^0 \in [0,1]$, $s_i^0 + e_i^0 +  p_i^0 + r_i^0 = 1$ for all $i \in [n] $. 
Then, for all $k \geq 0$ and $i \in [n]$, 
$s_i^k, e_i^k, p_i^k, r_i^k \in [0, 1]$ and $s_i^k + e_i^k + p_i^k + r_i^k = 1$. 
\end{lemma}
\begin{proof}
We prove this result by induction. 
By assumption, it holds for the base-case $k=0$. We follow the proof by showing the induction-step, that is, assume $s_i^k, e_i^k, p_i^k, r_i^k \in [0, 1]$ and $s_i^k + e_i^k + p_i^k + r_i^k = 1$, for all $i \in [n]$, and we now show that this holds also for time-step $k+1$.
By Assumption \ref{assm:seir} and \eqref{eq:seir_si}, $s^{k+1} \geq s_i^k + h\left[-s_i^k (\beta_i^E + \beta_i) \sum_{j\in \mathcal{N}_i}a_{ij}\right] = s_i^k \left[1 - h(\beta_i^E + \beta_i) \sum_{j\in \mathcal{N}_i}a_{ij}\right] \geq 0$. 
We have $s^{k+1} \leq s^k \leq 1$ since $
h\left[- s_i^k \left( \sum_{j\in \mathcal{N}_i} \beta^E_ia_{ij} e_j^k + \sum_{j\in \mathcal{N}_i} \beta_ia_{ij} p_j^k \right) \right] \leq 0$.
By Assumption \ref{assm:seir} and \eqref{eq:seir_ei}, $e_i^{k+1} \geq (1 - h\sigma_i) e_i^k \geq 0$. Moreover, by the assumption $e_j^k, p_j^k \leq 1$ for all $j \in [n]$, Assumption \ref{assm:seir}, and \eqref{eq:seir_ei}, $e_i^{k+1} \leq e_i^k + s_i^k h (\beta^E_i + \beta_i) \sum_{j\in \mathcal{N}_i}a_{ij} \leq e_i^k + s_i^k \leq 1$.
By Assumption \ref{assm:seir} and \eqref{eq:seir_pi}, $p_i^{k+1} \geq \left(1 - h\gamma_i \right) p_i^k \geq 0$ and $p_i^{k+1} \leq p_i^k + h\sigma_i e_i^k \leq p_i^k + e_i^k \leq 1$.
By Assumption \ref{assm:seir} and \eqref{eq:seir_ri}, $r_i^{k+1} \geq r_i^k \geq 0$, and $r_i^{k+1} \leq r_i^k + p_i^k$.

Thus, by the principle of mathematical induction we have that, if $s_i^0, e_i^0, p_i^0, r_i^0 \in [0, 1]$ and $s_i^0 + e_i^0 + p_i^0 + r_i^0 = 1$ for all $i \in [n]$ then  $s_i^k, e_i^k, p_i^k, r_i^k \in [0, 1]$ and $s_i^k + e_i^k + p_i^k + r_i^k = 1$ for all $k \in \mathbb{N}$.
\end{proof}


\section{Analysis of Models}\label{sec:analysis}

\vspace{-1ex}

In this section we discuss the limiting behavior of the networked models from Section \ref{sec:models}.

\vspace{-1.5ex}

\subsection{SIR Model}

\vspace{-1ex}

For the matrix $I + h \diag(s^k) BA - h \gamma$, define its dominant eigenvalue as  $\lambda_{max}$. 
\begin{theorem} \cite{sir_closed_loop}\label{thm:sir_dis}
Consider the model in \eqref{eq:nsir} with Assumption \ref{assm:0}, $BA$ irreducible, $s_i^0 > 0$ for all  $i \in [n]$,
and $p_i^0 > 0$ for some~$i$. 
Then, for all  $i \in [n]$, 
\begin{enumerate}
\item[1)] 
$s_i^{k+1} \leq s_i^k$, for all $k \geq 0$,
\item[2)] $\lim_{k \to \infty} p_i^k = 0$, 
\item[3)] 
$\lambda_{max}$ is monotonically decreasing as a function of $k$, 
\item[4)] there exists $\bar{k}$ such that $
\lambda_{max}< 1$ for all $k \geq \bar{k}$, and
\item[5)] there exists $\bar{k}$, such that $p_i^k$ converges linearly to $0$ for all~$k \geq~\bar{k}$. 
\end{enumerate}
\end{theorem}

\vspace{-1.5ex}

\subsection{SEIR Model}


Let $\lambda^{M_k}_{max}$ be the dominant eigenvalue of $M_k$, where $M_k$ is defined as 
\begin{align}\label{eq:M}
    M_k = \begin{bmatrix} (I + h {\rm diag}(s^k)B^EA - h\sigma) & \displaystyle h{\rm diag}(s^k) B A \\ h \sigma & \displaystyle (I - h\gamma) \end{bmatrix}.
\end{align}
Observe that $z^{k+1}:=\begin{bmatrix}e^{k+1} \\ p^{k+1}\end{bmatrix} = M_k \begin{bmatrix}e^k \\ p^k \end{bmatrix}$.

\begin{theorem}\label{prop:main}
Consider the model in \eqref{eq:seir_si}-\eqref{eq:seir_ri} under Assumption \ref{assm:seir}. Suppose $s_i^0, e_i^0, p_i^0,$ $r_i^0 \in [0,1]$, $s_i^0 + e_i^0 +  p_i^0 + r_i^0 = 1$ for all $i \in [n] $, 
$BA$ is irreducible, $s_i^0 > 0$ for all  $i \in [n]$, and $p_i^0 > 0$ for some~$i$. 
Then, for all $k \geq 0$ and $i \in [n]$, 
\begin{enumerate}
\item[1)] $s_i^{k+1} \leq s_i^k$,
\item[2)] $\lim_{k \to \infty} e_i^k = 0$ and $\lim_{k \to \infty} p_i^k = 0$,
\item[3)] $\lambda^{M_k}_{max}$ is monotonically decreasing as a function of k,
\item[4)] there exist a $\bar{k}$ such that $\lambda^{M_k}_{max} < 1$ for all $k \geq \bar{k}$,
\item[5)] there exists $\bar{k}$, such that $p_i^k$ converges linearly to $0$ for all $k \geq \bar{k}$ and $i \in [n]$.
\end{enumerate}
\end{theorem}

\begin{proof}
We present the proof for each part of the theorem, starting with 1).

1) By Lemma \ref{lem:seir} and Assumption \ref{assm:seir}, we have that $h\left[-s_i^k \left(\sum_{j\in \mathcal{N}_i} \beta_i^E a_{ij}e_j^k + \sum_{j=1}^k \beta_i a_{ij} p_j^k \right) \right] \leq 0$ for all $i \in [n]$ and $k \geq 0$. Therefore, from \eqref{eq:seir_si}, we have $s_i^{k+1} \leq s_i^k$.

2) Since the rate of change of $s^k$, $-h{\rm diag}(s^k)\left[B^EA e^k + BAp^k\right]$, is non-positive for all $k \geq 0$ and $s^k$ is lower bounded by zero, by Lemma \ref{lem:seir}, we conclude that $\lim_{k \to \infty}s^k$ exists. Therefore,
%
\begin{equation}
    \lim_{k\to \infty} - h {\rm diag}(s^k) \left[B^EAe^k + BAp^k\right] = \0. 
\end{equation}
%
%
Therefore, $\lim_{k \to \infty} e^{k+1} - e^k = \lim_{k \to \infty} -h \sigma e^k$. Thus, by Assumption \ref{assm:seir}, $h\sigma_i > 0$ for all $i \in [n]$, $\lim_{k \to \infty}e_i^k = 0$ for all $i \in [n]$. 

Similarly, we show that $\lim_{k \to \infty} p_i^k = 0$ for all $i \in [n]$. We have that $\lim_{k \to \infty} p^{k+1} - p^{k} = \lim_{k \to \infty}h\left( \sigma e^k - \gamma p^k\right) = \lim_{k \to \infty} - h \gamma p^k$, where we used that $\lim_{k \to \infty} e^k = 0$. By assumption $h\gamma_i > 0$ for all $i \in [n]$, thus $\lim_{k 
\to \infty} p_i^k = 0$ for all $i \in [n]$. 

3) By 
assumption  $s_i^0 > 0$ for all  $i \in [n]$, and from the proof of Lemma \ref{lem:seir} we can see that $s_i^k > 0$ for all  $i \in [n]$, $k\geq 0$. 
Therefore, since
we have that 
$BA$ is irreducible, from \eqref{eq:M}, the matrix $M_k$ is irreducible, and non-negative by Assumption \ref{assm:seir}, for all finite $k$.
Thus by the Perron-Frobenius Theorem for irreducible non-negative matrices we have that $\lambda^{M_k}_{max} = \rho(M_k)$. Since $\rho(M_k)$ increases when any entry increases \cite[Theorem 2.7]{varga} and by 1) of this theorem, we have that $\rho(M_k) \geq \rho(M_{k+1})$, that is $\lambda^{M_k}_{max} \geq \lambda^{M_{k+1}}_{max}$.

4) There are two possible equilibria: i) $\lim_{k \to \infty} s^k = \0 
$, and ii) $\lim_{k \to \infty} s^k = s^* \neq \0 
$. We explore the two cases separately.

i) If $\lim_{k \to \infty} s^k = \0 
$, 
\begin{equation*}
    \hspace{21ex} \lim_{k \to \infty} M^k = 
    \begin{bmatrix}  
        I-h\sigma & 0 \\ h\sigma & I-h\gamma
    \end{bmatrix}.
\end{equation*}
%
Therefore, by Assumption \ref{assm:seir},
there exists a $\bar{k}$ such that $\lambda^{M_k}_{max} < 1$ for all $k \geq \bar{k}$. 

ii) If $\lim_{k \to \infty} s^k = s^* \neq \0 
$, then, by 2), for any $\left(s^0, e^0, p^0, r^0\right)$ the system converges to some equilibrium of the form $\left(s^*, \0, \0, \1 - s^*\right)$. Define
%
\begin{equation}
    \epsilon_s^k := s^k - s^* \text{ and } \epsilon_p^k := 
    z^k
    - \0_{2n}. \label{eq:eps_def}
\end{equation}
%
By 1) and Lemma \ref{lem:seir}, respectively we know that $\epsilon_s^k \geq \0_n$ and $\epsilon_p^k \geq \0_{2n}$ for all $k \geq 0$. Furthermore, we know that $\epsilon_s^{k+1} \leq \epsilon_s^k$ for all $k \geq 0$, $\lim_{k \to \infty} \epsilon_s^k = \0_n$, and $\lim_{k \to \infty} \epsilon_p^k = \0_{2n}$. where the last equality comes from 2).

Linearizing the dynamics of $\epsilon_s^k$ and $\epsilon_p^k$ around $\left(s^*, \0_{2n}\right)$ gives
%
\begin{subequations} \label{eq:eps_sys}
\begin{eqnarray}
    &&\epsilon_s^{k+1} = \epsilon_s^k - h{\rm diag}(s^*)\begin{bmatrix}B^E A & B A \end{bmatrix}\epsilon_p^k, \\
    &&\epsilon_p^{k+1} = M_k \epsilon_p^k . \label{eq:eps_x}
\end{eqnarray}
\end{subequations}

Let $\lambda^{M^*}_{max}$ be the maximum eigenvalue of
%
\begin{equation}
    M^* = \begin{bmatrix} (I + h {\rm diag}(s^*)B^EA - h\sigma) & \displaystyle h{\rm diag}(s^*) B A \\ h \sigma & \displaystyle (I - h\gamma)  \end{bmatrix}
\end{equation}
%
with corresponding normalized eigenvector $w^*$, that is,
%
\begin{equation}
    {w^*}^\top M^* = \lambda^{M^*}_{max} {w^*}^\top.
\end{equation}

If $\lambda^{M^*}_{max} > 1$, then the system in \eqref{eq:eps_sys} is unstable. Therefore, by Lyapunov's Indirect Method, $\lim_{k \to \infty} \left(\epsilon_s^k, \epsilon_p^k \right) \neq \left( s^*, \0_{2n}\right)$, which is a contradiction.

Now consider the case where $\lambda^{M^*}_{max} = 1$. Define 

\begin{equation}
    \tilde{M}_k = \begin{bmatrix} h {\rm diag}(\epsilon_s^k)B^EA & h {\rm diag}(\epsilon_s^k)BA \\ 0 & 0 \end{bmatrix}.
\end{equation}
%
%
Then we can write $M_k = M^* + \tilde{M}_k$, observe that all entries in $\tilde{M}$ are non-negative.
Using \eqref{eq:eps_def} and left multiplying the equation of $\epsilon_p^{k+1}$ in \eqref{eq:eps_x} by ${w^*}^\top$ we get
%
\begin{align*}
    {w^*}^\top \epsilon_p^{k+1} 
    &= {w^*}^\top M \epsilon_p^k \\
    &= \lambda^{M^*}_{max} {w^*}^\top \epsilon_p^k + {w^*}^\top \tilde{M}_k \epsilon_p^k \\
    &= {w^*}^\top \epsilon_p^k + {w^*}^\top \tilde{M}_k \epsilon_p^k.
\end{align*}
%
Thus, 
%
\begin{equation}
    {w^*}^\top \left( \epsilon_p^{k+1} - \epsilon_p^k \right) = {w^*}^\top \tilde{M}_k \epsilon_p^k \geq 0,
\end{equation}
%
where the last inequality holds since all elements are non-negative. This contradicts that $\lim_{k \to \infty} 
z^k
= \0_{2n}$, that is 2). Therefore, there exists a $\bar{k}$ such that $\lambda^{M_k}_{max} < 1$ for all $k \geq \bar{k}$.

5) Since, by 4), there exists a $\bar{k}$ such that $\lambda^{M_k}_{max} < 1$ for all $k \geq \bar{k}$, and we know that $\lambda^{M_k}_{max} = \rho(M_k) \geq 0$ by Assumption \ref{assm:seir}, we have
%
\begin{equation}
    \lim_{k \to \infty} \frac{\| p^{k+1} \|}{\|p^k\|} = \frac{\|M_k p^k\|}{\|p^k\|} = \lambda^{M_k}_{max} < 1.
\end{equation}
%
Therefore, for $k \geq \bar{k}$, $p^k$ converges linearly to $\0_n$.
\end{proof}



\section{Estimating Model Parameters}\label{sec:parameters}


In this section we discuss estimating the parameters of the networked models from Section \ref{sec:models} using data. 
Note that these results are similar to the ones for the networked SIS model in \cite{dtjournal,ciss}. 


\subsection{Networked SIR Model}



We first explore conditions for estimating the model parameters for the homogeneous model, that is, where every node has the same infection and healing parameters. In order to do so we define the following matrices:
%
%
%
\begin{align}
    \Phi &= \begin{bmatrix} \displaystyle h
    S^{0}
    A p^{0}  & \displaystyle -hp^0 \\ \vdots & \vdots \\ 
    \undermat{\mathbf{a}}{\displaystyle h
    S^{T-1}
    A p^{T-1}}  & \undermat{-\mathbf{b}}{\displaystyle -hp^{T-1}} \end{bmatrix}, \text{ and} \label{eq:idPhi}
    \\[12pt]
    \Gamma &= \begin{bmatrix} 
    \0 &
    \mathbf{b}
    \end{bmatrix} \label{eq:idGamma}.
\end{align}
%
\noindent Using the above matrices we construct
$
    Q = \begin{bmatrix} \displaystyle \Phi \\ \displaystyle \Gamma \end{bmatrix}
    $
and, write \eqref{eq:pidiscrete} and \eqref{eq:ridiscrete} as 
\begin{equation}\label{eq:id1}
\hspace{3ex}    
\underbrace{\begin{bmatrix} \displaystyle p^1-p^0\\ \vdots \\ \displaystyle p^T-p^{T-1} \\ \displaystyle r^1 - r^0 \\ \vdots \\ \displaystyle r^T - r^{T-1} \end{bmatrix}}_{\Delta} =  Q \begin{bmatrix} \beta \\ \gamma \end{bmatrix}.
\end{equation}
%
\noindent We find the least squares estimates $\hat{\beta}$ and $\hat{\gamma}$ by employing the pseudoinverse of $Q$.
\begin{theorem} \label{thm:idsir}
Consider the model in \eqref{eq:nsir} with homogeneous virus spread, that is, $\beta$ and $\gamma$ are the same for all $n$ nodes. Assume that $s^k, p^k, r^k$, for all $k \in [T] \cup \{0\}$, and $h$ are known, with $n>0$. 
Then, the parameters of the spreading process can be identified uniquely if and only if $T > 0$, and there exist $i_1, i_2 \in [n]$ and $k_1, k_2 \in [T-1] \cup \{0\}$ such that
\begin{subequations} \label{eq:inv_sir}
\begin{align} 
 p_{i_1}^{k_1} &\neq 0, \label{eq:inv_gi_sir_homo} \\
 \left(S^{k_2} A p^{k_2} \right)_{i_2} &\neq 0.
\label{eq:inv_bi_sir_homo}
\end{align}
\end{subequations}
\end{theorem}

\begin{proof}
Using \eqref{eq:idPhi} and \eqref{eq:idGamma} we can write $Q$ as follows
%
%
\begin{align*}
    Q &= 
    \begin{bmatrix}
    \displaystyle \mathbf{a} & \displaystyle - \mathbf{b} \\
    \0 & \displaystyle \mathbf{b}
    \end{bmatrix}
    = \underbrace{\begin{bmatrix}
    \displaystyle I & \displaystyle -I \\
    \displaystyle \0 & \displaystyle I
    \end{bmatrix}}_{D}
    \underbrace{\begin{bmatrix}
    \displaystyle \mathbf{a} & \0 \\
     \0 & \mathbf{b}
    \end{bmatrix}}_{\tilde{Q}}.
\end{align*}
%
%
If the assumptions in \eqref{eq:inv_sir} hold, $\mathbf{a}$ and $\mathbf{b}$ each have at least one element that is nonzero, therefore $\tilde{Q}$ has full column rank. Clearly $D$ has full rank which implies that the rank for $Q$ is equal to the rank of $\tilde{Q}_i$ \cite{horn2012matrix}. Therefore, there exists a unique solution to \eqref{eq:id1} using the inverse or pseudoinverse.

If one of the assumptions in \eqref{eq:inv_sir} is not met, $Q$ will have a nontrivial nullspace. Therefore, in that case, \eqref{eq:id1} does not have a unique solution.
\end{proof}
Note that we do not need to know all $s^k, p^k, r^k$ but only the entries $s^k_j, p^k_j, r^k_j$, for $j\in \mathcal{N}_{i_1} \cup \mathcal{N}_{i_2} \cup \{i_1, i_2\}$ is sufficient, where $i_1,i_2$ satisfy \eqref{eq:inv_sir}. 



To estimate the spreading parameters for the discrete-time SIR model from Section \ref{subsec:nSIR} in the heterogeneous case, we form the following matrices:
%
%
\begin{align}
    \Phi_i &= \begin{bmatrix} \displaystyle h
    s_{i}^{0}
    \sum_{j\in \mathcal{N}_i} a_{ij}p_{j}^{0}  & \displaystyle -hp_{i}^0 \\ \vdots & \vdots \\ 
    \undermat{\mathbf{a}_i}{\displaystyle h
    s_{i}^{T-1}
    \sum_{j\in \mathcal{N}_i} a_{ij}p_{j}^{T-1}}  & \undermat{-\mathbf{b}_i}{\displaystyle -hp_{i}^{T-1}} \end{bmatrix}, \text{ and} \label{eq:idPhii}
    \\[12pt]
    \Gamma_i &= \begin{bmatrix} 
    \0 &
    \mathbf{b}_i
    \end{bmatrix} \label{eq:idGammai}.
\end{align}
%
Using the above matrices we construct
$
  Q_i = \begin{bmatrix} \displaystyle \Phi_i \\ \displaystyle \Gamma_i \end{bmatrix}
  $
%
%

\begin{equation}\label{eq:idi1}
\hspace{3ex}    \underbrace{\begin{bmatrix} \displaystyle p_{i}^1-p_{i}^0\\ \vdots \\ \displaystyle p_{i}^T-p_{i}^{T-1} \\ \displaystyle r_{i}^1 - r_{i}^0 \\ \vdots \\ \displaystyle r_{i}^T - r_{i}^{T-1} \end{bmatrix}}_{\Delta_i} =  Q_i \begin{bmatrix} \beta_{i} \\ \gamma_{i} \end{bmatrix}.
\end{equation}
%
For each $i$, we find the least squares estimates $\hat{\beta}_i$ and $\hat{\gamma}_i$ by using the pseudoinverse of $Q_{i}$. 

We now explore conditions for estimating the SEIR model parameters in the homogeneous case. In order to do so we define the following matrices:
\begin{align}
    \Phi^E &= \begin{bmatrix} \displaystyle h
    S^{0}
    A e^{0}  & \displaystyle h
    S^{0}
    A p^{0}  & \displaystyle -h {e}^0 & 
    \0 \\ \vdots & \vdots & \vdots & \vdots 
    \\
    \undermat{\mathbf{a}^E}{\displaystyle h
    S^{T-1} A e^{T-1}}  & \undermat{\mathbf{b}^E}{\displaystyle h
    S^{T-1}
    A p^{T-1}}  & \displaystyle -h{e}^{T-1} & \0 
    \end{bmatrix}, \label{eq:idPhi_SEIR}
    \\[15pt] 
    \Sigma^E &= \begin{bmatrix} \displaystyle 0 & \displaystyle 0 & \displaystyle h e^{0}  & \displaystyle -h p^0\\ \vdots & \vdots & \vdots & \vdots \\ \displaystyle 0 & \displaystyle 0 & \undermat{\mathbf{c}^E}{\displaystyle h e^{T-1}}  & \undermat{-\mathbf{d}^E}{\displaystyle -h p^{T-1}} \end{bmatrix}, \label{eq:idSigma_SEIR}\\ 
    \text{ and} \nonumber 
    \\
    \Gamma^E &= \begin{bmatrix}  
    \0 & \0 & \0 &
    \mathbf{d}^E
    \end{bmatrix} \label{eq:idGamma_SEIR}.
\end{align}
%
Using the above matrices we construct
%
$
    Q^E = \begin{bmatrix} \displaystyle \Phi^E \\ \displaystyle \Sigma^E \\ \displaystyle \Gamma^E \end{bmatrix}
    $
%
and
write \eqref{eq:nseir}-\eqref{eq:iota1} as 
%
\begin{equation}\label{eq:id2}
\hspace{3ex}
    \begin{bmatrix} 
    \displaystyle e^1-e^0\\ \vdots \\ \displaystyle e^T-e^{T-1} \\ 
    \displaystyle p^1-p^0\\ \vdots \\ \displaystyle p^T-p^{T-1} 
    \\ 
    \displaystyle r^1 - r^0 \\ \vdots \\ \displaystyle r^T - r^{T-1} 
    \end{bmatrix} =  Q^E \begin{bmatrix} \beta^E \\ \beta \\ \sigma \\ \gamma \end{bmatrix}. 
\end{equation}
%
We find the least squares estimates $\hat{\beta}^E$, $\hat{\beta}$, $\hat{\sigma}$, and $\hat{\gamma}$ using the pseudoinverse of $Q^E$.

\begin{theorem} \label{thm:idseir}
Consider the model in \eqref{eq:nseir}-\eqref{eq:iota1} with homogeneous virus spread, that is, $\beta^E$, $\beta$, $\sigma$, and $\gamma$ are the same for all $n$ nodes.
Assume that $s^{k},e^{k},p^{k},r^{k}$, for all $k\in[T] \cup \{0\}$, and $h$ are known, with $n > 1$. Then, the parameters of the spreading process  can be identified uniquely if and only if $T>0$, and
there exist $i_1, i_2, i_3, i_4 \in [n]$ and $ k_1,k_2,k_3,k_4\in[T-1] \cup \{0\}$ such that
\begin{subequations} \label{eq:inv_seir}
\begin{align} 
& p_{i_1}^{k_1} \neq 0, \label{eq:inv_g}
e_{i_2}^{k_2} \neq 0, 
\\
& g^{k_3}_{i_3}(e^{k_3}) g^{k_4}_{i_4}(p^{k_4}) \neq g^{k_4}_{i_4}(e^{k_4}) g^{k_3}_{i_3}(p^{k_3}),
\label{eq:inv_b}
\end{align}
\end{subequations}
%
where $g^k_i(x) = s_i^{k} \sum_{j\in \mathcal{N}_i} a_{ij}x_{j}$.
\end{theorem}

\begin{proof}
Using \eqref{eq:idPhi_SEIR}-\eqref{eq:idGamma_SEIR}, we can write $Q^E$ as follows

\begin{equation}
    Q^E= 
    \underbrace{\begin{bmatrix}
    \displaystyle I & \displaystyle -I & \displaystyle \0_{_{nT \times nT}} \\ \0_{_{nT \times nT}} & \displaystyle I & \displaystyle -I \\ \0_{_{nT \times nT}} & \0_{_{nT \times nT}} & I 
    \end{bmatrix}}_{D^E} 
    \underbrace{\begin{bmatrix} \displaystyle \mathbf{a}^E  & \displaystyle \mathbf{b}^E  & \0 & 
    \0 \\ \0  & \0  & \displaystyle \mathbf{c}^E & \0 \\ \0  & \0  & \0 & \displaystyle \mathbf{d}^E
    \end{bmatrix}}_{\tilde{Q}^E}.
\end{equation}
%
Since $n>1$, $\tilde{\Phi}^E = \begin{bmatrix}\mathbf{a}^E & \mathbf{b}^E \end{bmatrix}$ has at least two rows, and given that \eqref{eq:inv_b} holds, $\tilde{\Phi}^E$ has column rank equal to two. Moreover, if \eqref{eq:inv_g} 
holds
$\mathbf{c}$ and $\mathbf{d}$ each have at least one element that is nonzero. Thus, $\tilde{Q}^E$ has full column rank. Clearly $D^E$ has full rank which implies that the rank of $Q^E$ is equal to the rank of $\tilde{Q}^E$ \cite{horn2012matrix}. Therefore, there exists a unique solution to \eqref{eq:id2} using the pseudoinverse.

If one of the assumptions in \eqref{eq:inv_g}-\eqref{eq:inv_b} is not met, $Q^E$ will have a nontrivial nullspace. Therefore, in that case, \eqref{eq:id2} does not have a unique solution.
\end{proof}
%
Similar to the SIR model and the heterogeneous case it is not necessary to  know all entries of 
$s^k, e^k, p^k, r^k$. 
It is sufficient to know only $s^k_j, e^k_j, p^k_j, r^k_j$, for $j \in \mathcal{N}_{i_1} \cup \mathcal{N}_{i_2} \cup \mathcal{N}_{i_3} \cup \mathcal{N}_{i_4} \cup \{i_1, i_2, i_3, i_4\}$, where $i_1,i_2,i_3,i_4$ satisfy~\eqref{eq:inv_seir}.


To estimate the spreading parameters for the discrete-time, heterogeneous SEIR model from Section \ref{sec:nSEIR} we define: 
%
%
\begin{align}
    \Phi^E_i &= \begin{bmatrix} \displaystyle h
    s_{i}^{0}
    \sum_{j\in \mathcal{N}_i} a_{ij}e_{j}^{0}  & \displaystyle h
    s_{i}^{0}
    \sum_{j\in \mathcal{N}_i} a_{ij}p_{j}^{0}  & \displaystyle -h {e}_{i}^0 & 
    0 \\ \vdots & \vdots & \vdots & \vdots 
    \\ 
    \undermat{\mathbf{a}^E_i}{\displaystyle h
    s_{i}^{T-1}\sum_{j\in \mathcal{N}_i} a_{ij}e_{j}^{T-1}}  & \undermat{\mathbf{b}^E_i}{\displaystyle h
    s_{i}^{T-1}
    \sum_{j\in \mathcal{N}_i} a_{ij}p_{j}^{T-1}}  & \displaystyle -h{e}_{i}^{T-1} & 0 
    \end{bmatrix}, \label{eq:idPhii_SEIR}
    \\[5pt] 
    \Sigma^E_i &= \begin{bmatrix} \displaystyle 0 & \displaystyle 0 & \displaystyle h e_{i}^{0}  & \displaystyle -h p_{i}^0\\ \vdots & \vdots & \vdots & \vdots \\ \displaystyle 0 & \displaystyle 0 & \undermat{\vspace{1ex}\mathbf{c}^E_i}{\displaystyle h e_{i}^{T-1}}  & \undermat{-\mathbf{d}^E_i}{\displaystyle -h p_{i}^{T-1}} \end{bmatrix}, \label{eq:idSigmai_SEIR}\\ 
    \text{ and} \nonumber
    \\[3pt]
    \Gamma^E_i &= \begin{bmatrix}  
    \0 & \0 & \0 & \mathbf{d}^E_i
\end{bmatrix} \label{eq:idGammai_SEIR}.
\end{align}
%
Using the above matrices we construct
%
$
    Q^E_i = \begin{bmatrix} \displaystyle \Phi^E_i \\ \displaystyle \Sigma^E_i \\ \displaystyle \Gamma^E_i \end{bmatrix}
    $
%
and
write \eqref{eq:nseir}-\eqref{eq:iota1} as 

\begin{equation}\label{eq:id2i}
\hspace{-3ex}
    \begin{bmatrix} \displaystyle e_{i}^1-e_{i}^0\\ \vdots \\ \displaystyle e_{i}^T-e_{i}^{T-1} \\ 
    \displaystyle p_{i}^1-p_{i}^0\\ \vdots \\ \displaystyle p_{i}^T-p_{i}^{T-1} \\ \displaystyle r_{i}^1 - r_{i}^0 \\ \vdots \\ \displaystyle r_{i}^T - r_{i}^{T-1} 
    \end{bmatrix} =  Q^E_i \begin{bmatrix} \beta^E_{i} \\ \beta_{i} \\ \sigma_i \\ \gamma_{i} \end{bmatrix}. 
\end{equation}
%
We find the least squares estimates $\hat{\beta}^E_{i}$, $\hat{\beta}_{i}$, $\hat{\sigma}_{i}$, and $\hat{\gamma}_{i}$ using the pseudoinverse of $Q^E_{i}$.

\begin{theorem} \label{thm:idiseir}
Consider the model in \eqref{eq:nseir}-\eqref{eq:iota1}.
Assume that $s_i^{k},e_j^{k},p_j^{k},r_i^{k}$, for all $ j\in \mathcal{N}_i \cup \{i\}, k\in[T-1] \cup \{0\}$, $e_i^{T},p_i^{T},r_i^{T}$, and $h$ are known. Then, the parameters of the spreading process for node $i$ can be identified uniquely if and only if $T>1$, and
there exist $ k_1,k_2,k_3,k_4\in[T-1] \cup \{0\}$ such that
\begin{subequations} \label{eq:inv_seiri}
\begin{align} 
& p_i^{k_1} \neq 0, \label{eq:inv_gi} 
e_i^{k_2} \neq 0, 
\\
& g^{k_3}_{i}(e^{k_3}) g^{k_4}_{i}(p^{k_4}) \neq g^{k_4}_{i}(e^{k_4}) g^{k_3}_{i}(p^{k_3}),
\label{eq:inv_bi}
\end{align}
\end{subequations}
%
where $g^k_i(x) = s_i^{k} \sum_{j\in \mathcal{N}_i} a_{ij}x_{j}$ which only uses the entries $x_j$ for which $j \in \mathcal{N}_i$.
\end{theorem}

\begin{proof}
Using \eqref{eq:idPhii_SEIR}-\eqref{eq:idGammai_SEIR}, we can write $Q^E_i$ as follows
\begin{equation*}
    Q^E_i = \underbrace{\begin{bmatrix}
    \displaystyle I & \displaystyle -I & \displaystyle \0_{_{T \times T}} \\ \0_{_{T \times T}} & \displaystyle I & \displaystyle -I \\ \0_{_{T \times T}} & \0_{_{T \times T}} & I 
    \end{bmatrix}}_{D^E_i}
    \underbrace{\begin{bmatrix} \displaystyle \mathbf{a}^E_i  & \displaystyle \mathbf{b}^E_i  & \0 & 
    \0 \\ \0  & \0  & \displaystyle \mathbf{c}^E_i & \0 \\ \0  & \0  & \0 & \displaystyle \mathbf{d}^E_i
    \end{bmatrix}}_{\tilde{Q}^E_i}.
\end{equation*}
%
Since $T>1$, $\tilde{\Phi}_i^E = \begin{bmatrix}\mathbf{a}^E_i & \mathbf{b}^E_i \end{bmatrix}$ has at least two rows, and given that \eqref{eq:inv_bi} holds, $\tilde{\Phi}^E_i$ has column rank equal to two. Moreover, if \eqref{eq:inv_gi} 
holds, 
$\mathbf{c}$ and $\mathbf{d}$ each have at least one element that is nonzero. Thus, $\tilde{Q}^E_i$ has full column rank. Clearly $D^E_i$ has full rank which implies that the rank of $Q^E_i$ is equal to the rank of $\tilde{Q}^E_i$ \cite{horn2012matrix}. Therefore, there exists a unique solution to \eqref{eq:id2} using the pseudoinverse.

If one of the assumptions in \eqref{eq:inv_gi}-\eqref{eq:inv_bi} is not met, $Q_i^E$ will have a nontrivial nullspace. Therefore, in that case, \eqref{eq:id2i} does not have a unique solution.
\end{proof}











\section{Simulations}\label{sec:sim}



\begin{figure}
\centering
\begin{subfigure}{0.3\textwidth}
     \centering
    \begin{overpic}[
    width = 1\textwidth]{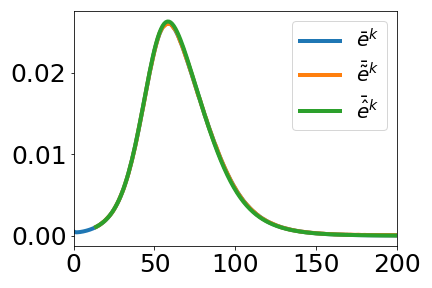}
            \put(52.5,-3){{\parbox{\linewidth}{%
                k}}}
        \end{overpic}
     \caption{ 
    Network
     average exposed state.}
    \label{fig:sim_e}
\end{subfigure}
\begin{subfigure}{0.3\textwidth}
    \centering
    \begin{overpic}[
    width = 1\textwidth]{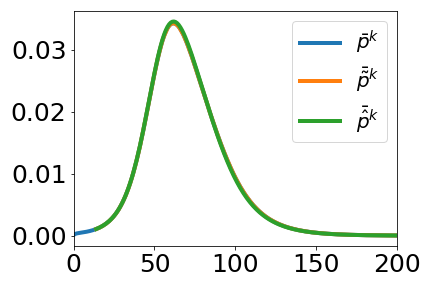}
            \put(52.5,-3){{\parbox{\linewidth}{%
                k}}}
        \end{overpic}
    \caption{ 
    Network 
    average infected state.}
    \label{fig:sim_p}
   \end{subfigure}
\begin{subfigure}{0.3\textwidth}
    \centering
        \begin{overpic}[width = 1\textwidth]{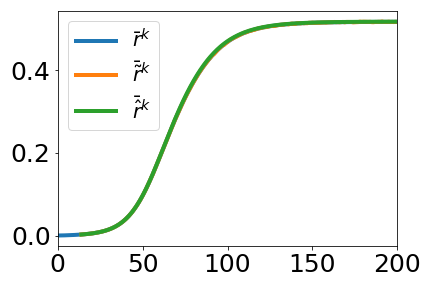}
            \put(51.5,-3){{\parbox{\linewidth}{%
                k}}}
        \end{overpic}
    \caption{ 
    Network 
    average removed state. 
     }
    \label{fig:sim_r}
   \end{subfigure}    
    \caption{Simulation of a homogeneous SEIR system with its measured states, and the recovered states 
    to show 
    how well the recovered states captures the 
    average state of the system.
    }
    \label{fig:sim_seir_avg}
\end{figure}



In this section we illustrate the analysis and parameter estimation results from Sections \ref{sec:analysis}-\ref{sec:parameters}.
For the adjacency matrix we use the nearest-neighbor network of counties in the northeast US with self-loops, namely, the counties in  Massachusetts~(MA), New Jersey~(NJ), Rhode Islands~(RI),  Connecticut~(CT), and New York~(NY), combining the five counties that make up New York City into one.


To simulate the states for the SEIR model we use \eqref{eq:nseir}-\eqref{eq:iota1}, with the following spread parameters $\left( \beta^{E}, \, \beta, \, \sigma, \, \gamma \right) = \left(0.04, \, 0.06, \, 0.4, \, 0.3 \right)$ and the initial state $e^0_1 = 0.02$, $e^0_2=0.03$, $p^0_1=0.01$, with the rest of the initial conditions for the non-susceptible states are set to zero for each node.
We correctly recover the spread parameters using \eqref{eq:id2} and $e^k$, $p^k$, and $r^k$ for $k \in \{0, 1\}$, as expected by Theorem \ref{thm:idseir}. 
Similarly, we correctly recover the spread parameters for the SIR model using  $\beta$, $\gamma$, $p^0$, and $r^0$ to simulate $p^1$ and $r^1$ and recovering the spread parameters using \eqref{eq:id1} and $p^k$ and $r^k$ for $k \in \{0, 1\}$.

We add measurement noise to evaluate the sensitivity of the estimation results and assume that the perturbation on $e$ is greater than that on $p$ and $r$ since it is the most difficult 
of the three states 
to measure. The measured states are $\tilde{e}$, $\tilde{p}$, and $\tilde{r}$,  determined by $\tilde{e}^k_i = e^k_i + \varepsilon_e(e^k_i)$ where $\varepsilon_e(x_i) \sim \mathcal{N}(0, 0.015x_i + 0.0001)$, $\tilde{p}^k_i = p^k_i + \varepsilon(p^k_i)$, and $\tilde{r}^k_i = r^k_i + \varepsilon(r^k_i)$ where $\varepsilon(x_i) \sim \mathcal{N}(0, 0.008x_i + 0.00001)$. In order to emulate the difficulty of measuring the states at the beginning of an outbreak, we start measuring from $k = 14$,
and recover the spread parameters by left multiplying \eqref{eq:id2} by the pseudo-inverse of $Q^E$. The estimated states $\hat{e}$, $\hat{p}$, and $\hat{r}$ are constructed using \eqref{eq:nseir}-\eqref{eq:iota1}, the first set of measured states $\tilde{e}^{14}$, $\tilde{p}^{14}$, and $\tilde{r}^{14}$, and the recovered spread parameters. In Figure \ref{fig:sim_seir_avg} we show how well the average states are recovered compared to the average of the actual states, $e$, $p$, and $r$ using the measured states to recover the spread parameters. The recovered spread parameters are $\left( \hat{\beta}^{E}, \, \hat{\beta}, \, \hat{\sigma}, \, \hat{\gamma} \right) = \left(0.0398, \, 0.0602, \, 0.4000, \, 0.3000 \right)$. The error of $\hat{e}$, $\hat{p}$, and $\hat{r}$ are $0.0190$, $0.0186$, and $0.0055$, respectively.

\section{Conclusion}\label{sec:con}


In conclusion, the discrete time SIR and SEIR models are extended to capture virus spread on a network. The limiting behavior of each model is analyzed and sufficient conditions for estimating spread parameters from data are presented. The developed models are implemented in simulations. To extend this work and improve the performance of the model, we plan to incorporate asymptomatic transmission and apply the results to real data.





\bibliography{bib}

\begin{thebibliography}{10}
\providecommand{\url}[1]{#1}
\csname url@samestyle\endcsname
\providecommand{\newblock}{\relax}
\providecommand{\bibinfo}[2]{#2}
\providecommand{\BIBentrySTDinterwordspacing}{\spaceskip=0pt\relax}
\providecommand{\BIBentryALTinterwordstretchfactor}{4}
\providecommand{\BIBentryALTinterwordspacing}{\spaceskip=\fontdimen2\font plus
\BIBentryALTinterwordstretchfactor\fontdimen3\font minus
  \fontdimen4\font\relax}
\providecommand{\BIBforeignlanguage}[2]{{%
\expandafter\ifx\csname l@#1\endcsname\relax
\typeout{** WARNING: IEEEtran.bst: No hyphenation pattern has been}%
\typeout{** loaded for the language `#1'. Using the pattern for}%
\typeout{** the default language instead.}%
\else
\language=\csname l@#1\endcsname
\fi
#2}}
\providecommand{\BIBdecl}{\relax}
\BIBdecl

\bibitem{lewnard2020scientific}
J.~A. Lewnard and N.~C. Lo, ``Scientific and ethical basis for
  social-distancing interventions against {COVID}-19,'' \emph{The Lancet
  Infectious Diseases}, 2020.

\bibitem{sir}
W.~O. Kermack and A.~G. McKendrick, ``A contribution to the mathematical theory
  of epidemics,'' \emph{Proceedings of the Royal Society A}, vol. 115, no. 772,
  pp. 700--721, 1927.

\bibitem{seir}
M.~Li and J.~Muldowney, ``Global stability for the {SEIRS} model in
  epidemiology,'' \emph{Math. Biosciences}, vol. 125, no.~2, pp. 155--164,
  1995.

\bibitem{mei2017epidemic}
W.~Mei, S.~Mohagheghi, S.~Zampieri, and F.~Bullo, ``On the dynamics of
  deterministic epidemic propagation over networks,'' \emph{Annual Reviews in
  Control}, vol.~44, pp. 116 -- 128, 2017.

\bibitem{sir_closed_loop}
A.~R. Hota, J.~Godbole, P.~Bhariya, and P.~E. Par\'{e}, ``A closed-loop
  framework for inference, prediction and control of {SIR} epidemics on
  networks,'' \emph{submitted to Annual Reviews in Control}, 2020,
  \url{https://arxiv.org/abs/2006.16185}.

\bibitem{zhou2011analysis}
X.~Zhou and J.~Cui, ``Analysis of stability and bifurcation for an {SEIR}
  epidemic model with saturated recovery rate,'' \emph{Comm. in Nonlinear
  Science and Numerical Sim.}, vol.~16, no.~11, pp. 4438--4450, 2011.

\bibitem{groendyke2020modifying}
C.~Groendyke and A.~Combs, ``Modifying the network-based stochastic {SEIR}
  model to account for quarantine,'' \emph{arXiv preprint arXiv:2008.01202},
  2020.

\bibitem{arcede2020accounting}
J.~P. Arcede, R.~L. Caga-anan, C.~Q. Mentuda, and Y.~Mammeri, ``Accounting for
  symptomatic and asymptomatic in a {SEIR}-type model of {COVID-19},''
  \emph{arXiv preprint arXiv:2004.01805}, 2020.

\bibitem{varga}
R.~S. Varga, \emph{Matrix Iterative Analysis}.\hskip 1em plus 0.5em minus
  0.4em\relax Springer-Verlag, 2000.

\bibitem{dtjournal}
P.~E. Par\'{e}, J.~Liu, C.~L. Beck, B.~E. Kirwan, , and T.~Ba\c{s}ar,
  ``Discrete time virus spread processes: Analysis, identification, and
  validation,'' \emph{IEEE Transactions on Control Systems Technology},
  vol.~28, no.~1, pp. 79--93, 2020.

\bibitem{ciss}
D.~Vrabac, P.~E. Par\'{e}, H.~Sandberg, and K.~H. Johansson, ``Overcoming
  challenges for estimating virus spread dynamics from data,'' in \emph{54th
  Annual Conference on Information Sciences and Systems (CISS)}, 2020.

\bibitem{horn2012matrix}
R.~A. Horn and C.~R. Johnson, \emph{Matrix Analysis}.\hskip 1em plus 0.5em
  minus 0.4em\relax Cambridge University Press, 2012.

\end{thebibliography}

\end{document}